\documentclass[oneside,11pt,a4paper]{article}
\usepackage{amsmath}
\usepackage{amsfonts}
\usepackage{amsthm}
\usepackage{amssymb}
\usepackage{color}
\usepackage[dvips]{graphicx}
\makeatletter
\renewcommand{\section}{\@startsection{section}{1}
{\parindent}{3.5ex plus 1ex minus .2ex} {2.3ex plus
.2ex}{\bf\large}}
\renewcommand{\subsection}{\@startsection{subsection}{2}
{\parindent}{3.5ex plus 1ex minus .2ex} {2.3ex plus .2ex}{\large}}
\newtheorem{thm}{Theorem}[section]

\newtheorem{lm}{Lemma}[section]
\newtheorem{de}{Definition}[section]

\newtheorem{co}{Corrolary}[section]

\renewcommand{\l@subsection}{\@dottedtocline{4}{}{}}

\renewcommand{\tableofcontents}%
{\section*{\contentsname}\@starttoc{toc}}

\renewcommand{\@biblabel}[1]{#1.\hfill} \makeatother

\usepackage{epsfig}

\begin{document}

\title{On a class of distributions stable under random summation}

\author{L.B. Klebanov\footnotemark[1] \,\,,\,A.V. Kakosyan\footnotemark[2] \,\,,\,S.T. Rachev\footnotemark[3]
\,\,,\,G. Temnov \footnotemark[4] }
%               Tel.: +431 58801-10534 \\
%               mailto:gtemnov@fam.tuwien.ac.at}

\footnotetext[1]{{\bf Affiliation}:\, Department of Probability and
Statistics of Charles University, Prague Sokolovska 83, Prague-8, CZ
18675, Czech
Republic\,.\,\,Email:\,klebanov@chello.cz\\
This work was supported by grants MSM 002160839 and IAA 101120801  }

\footnotetext[2]{{\bf Affiliation}: Yerevan State University,
Yerevan, Armenia  }

\footnotetext[3]{{\bf Affiliation}: University of Karlsruhe, Germany
}

\footnotetext[4]{Corresponding author; {\bf email:
g.temnov@ucc.ie}\,\\{\bf Affiliation}: School of Mathematical
Sciences, University College Cork. \,\,\,Support by SFI research
grant 07/MI/008 associated with Edgeworth Centre for Financial
Mathematics}

\date{\today}
\maketitle

\begin{abstract}
We investigate a family of distributions having a property of
stability-under-addition, provided that the number $\nu$ of added-up
random variables in the random sum is also a random variable. We
call the corresponding property a \,$\nu$-stability and investigate
the situation with the semigroup generated by the generating
function of $\nu$ is commutative.

Using results from the theory of iterations of analytic functions,
we show that the characteristic function of such a $\nu$-stable
distribution can be represented in terms of Chebyshev polynomials,
and for the case of $\nu$-normal distribution, the resulting
characteristic function corresponds to the hyperbolic secant
distribution.

We discuss some specific properties of the class and present
particular examples.
\\[0.5cm]
\end{abstract}
\textbf{Key Words:} Stability, random summation, characteristic
function, hyperbolic secant distribution.

%\paragraph{abstract}. MSC classes: 60E07, 60E10, 62E10.

%\footnotetext[2]{This work has been done under the financial
%support by Christian Doppler Labor }
%\url{http://www.prismalab.at}}

%\paragraph{Abstract}
%A practical overview

%\vspace*{0.4cm}

%\paragraph{AMS--classification} 62\,F\,12 (primary), 91\,B\,70

%\vspace*{0.3cm}

%\paragraph{Key words} ML estimation, consistency, Pareto distribution, Mis-specified model, inflation

%\vspace*{0.3cm}

%\tableofcontents

\section{Introduction}

In many applications of probability theory certain specific
classes of distributions have become very useful, usually called
"fat tailed" of "heavy tailed" distributions. The {\it Stable
distributions} that originate from the Central Limit problem, are
probably most popular among the heavy tailed distributions,
however there is a wide collection of classes of distributions,
all related to Stable ones in many various ways, often these
relations are not at all obvious.

Besides, certain generalizations of stable distributions are known,
using sums of random numbers of random variables (instead of sums
with deterministic number of summands), see e.g. Gnedenko
\cite{Gn1983}, Klebanov, Mania, Melamed \cite{KlebanovMM}, for the
examples of such, including the so-called {\it $\nu$-stable}
distributions, introduced independently by Klebanov and Rachev
\cite{KlRa} and Bunge \cite{Bung}.

In the present paper, we focuse on presenting further examples of
strictly $\nu$-stable random variables, that could be useful in
practical applications, including applications in financial
mathematics.

%The rest of the paper is organized as follows "?"

\section{Definition of strictly $\nu$-stable r.v.'s, properties and examples}

In the present section, we give a general insight on strictly
$\nu$-stable distributions and describe some examples that have
been mentioned in the literature before.

\subsection{Basic definitions}

Let $X,X_1,X_2,\dots,X_n,\dots$ be a sequence of i.i.d. random
variables, and let $\{\,\nu_p\,,\,p\in\Delta\,\}$ be a family of
some discrete r.v.'s taking values in the set of natural numbers
$\mathbb{N}$. Assume that this family does not depend on the
sequence $\{X_j,\,j\geq1\}$, and that, for $\Delta\subset(0,1)$,
\begin{equation}
\mathbf{E}\,\left[\,\nu_p\,\right]=\frac{1}{p}\,,\,\,\,\forall\,p\in\Delta\,.
\end{equation}

\begin{de}
We say that the r.v. $X$ has a strictly $\nu$--stable
distribution, if \,\,$\forall\,\,p\in\Delta$ it holds that
\[
X\overset{d}{=}p^{1/\alpha}\sum_{i=1}^{\nu_p}X_j\,,
\]
where $\alpha\in(0,2]$ is called the index of stability.
\end{de}

After this general definition, a narrower class is defined for
$\alpha = 1/2$.
\begin{de}
We call the r.v. $X$ a strictly $\nu$--normal r.v., if\,\,
$\mathbf{E}X=0$, $\mathbf{E}X^2=\infty$, and the following holds:
\[
X\overset{d}{=}p^{1/2}\sum_{i=1}^{\nu_p}X_j\,,\,\,\,\,\forall\,\,p\in\Delta\,\,.
\]
\end{de}

Closely related to the stability property is the property of
infinite divisibility, so we also give the following definition.
\begin{de}
$X$ has a strictly $\nu$--infinitely divisible distribution, if for
any $p\in\Delta$, there exists a r.v. $Y^{(p)}$, s.t.
\[
X\overset{d}{=}\sum_{j=1}^{\nu_p}Y^{(p)}_j\,,\,\,\,\,\mbox{with}\,
\,Y^{(p)},Y^{(p)}_1,\dots,Y^{(p)}_n,\dots\,\,\mbox{being iid
r.v.'s}\,\,.
\]
\end{de}

A powerful tool for investigating distributions' properties is the
{\it generating function}. We shall use the generating function of
the r.v. $\nu_p$\,\, denoting it by
$\mathcal{P}_p(z):=\mathbf{E}\left[z^{\nu_p}\right]$. Moreover, we
denote by $\mathcal{A}$ the semigroup generated by the family
$\{\,\mathcal{P}_p\,,\,\,p\in\Delta\}$, with the operation of the
functions' composition.

\subsection{Summary of the known results}

With regards to the definitions above, the following results are
known (see e.g. \cite{HeavyTailed} for proofs and details).

\begin{thm}
For the family $\{\,\mathcal{P}_p\,,\,\,p\in\Delta\}$, with
\,\,$\mathbf{E}\left[\nu_p\right]=\frac{1}{p}$\,,\,\, there exists
a strictly $\nu$-normal distribution, iff \,\,the semigroup
$\mathcal{A}$ is commutative.
\end{thm}

Suppose that we have a commutative semigroup $\mathcal{A}$. Then the
following statements (that we refer to in the sequel as {\it
Properties}) are known to be true (see \cite{HeavyTailed} for proofs
and details):
\begin{enumerate}
\item
The system
\begin{equation}\label{O3}
\varphi(t)=\mathcal{P}_p(\varphi(pt)),\,\,\,\forall\,p\in\Delta\,.
\end{equation}
of functional equations has a solution that satisfies the initial
conditions
\begin{equation}\label{O4}
\varphi(0)=1\,,\,\,\,\varphi^{\prime}(0)=-1\,.
\end{equation}
The solution is unique. In addition, there exists a distribution
function (cdf) $A(x)$ \,\,(with $A(0)=0$)\,such that
\begin{equation}\label{O5}
\varphi(t)=\int\limits_0^\infty e^{-tx}dA(x).
\end{equation}

\item
The characteristic function (ch.f.) of the strictly $\nu$-normal
distribution has the form
\begin{equation}\label{O6}
f(t)=\varphi(at^2)\,,\,\,\,\,a>0\,.
\end{equation}

\item
A ch.f. $g(t)$ is a ch.f. of a $\nu$-infinitely divisible r.v., iff
there exists a chf $h(t)$ of an infinitely divisible (in the usual
sense) r.v., such that
\begin{equation}\label{InfDiv}
f(t)=\varphi(-\ln h(t)).
\end{equation}

\end{enumerate}

The relation (\ref{InfDiv}) allows obtaining explicit
representations of ch.f. of strictly $\nu$-stable distributions.
Clearly, they are obtained through applying (\ref{InfDiv}) to a
ch.f. \,$h(t)$\,, provided that the r.v. corresponding to
\,$h(t)$\, is strictly stable (in the usual sense). Moreover, note
that the ch.f. $\varphi(ait)$,\, $a\in\mathbb{R}^1$, is the ch.f.
of an analogue of the degenerate r.v., and that for the r.v. with
such ch.f. the following analogue of the Law of Large Numbers
exists.

\begin{thm}
Let $X_1,X_2,\dots,X_n,\dots$ be a sequence of iid random
variables with the finite absolute value of the first moment, and
$\{\,\nu_p\,,\,p\in\Delta\,\}$ a family of r.v.'s taking values in
$\mathbb{N}$, independent of the sequence $\{X_j,j=1,2,\dots\}$.
Assume that $\mathbf{E}\left[\nu_p\right]=\frac{1}{p}$ and that
the semigroup $\mathcal{A}$ is commutative. \par Then the series
$p\sum\limits_{j=1}^{\nu_p}X_j$ is convergent is distribution, as
$p\rightarrow0$, and the limit of convergence is a r.v. having the
ch.f.\,\, $\varphi(ait)$.
\end{thm}

The proof of this theorem follows straightforwardly from the {\it
Property 1} outlined above and from the {\it Transfer Theorem} of
Gnedenko, see e.g. \cite{GnedKor}.

\vspace*{0.5cm} In the following paragraph we discuss several
particular examples of strictly $\nu$-normal and strictly
$\nu$-stable distributions.

\vspace*{-0.2cm}
\subsection{Examples and the outline of the problem}

{\bf Example 1\,.}\; {\it The usual stability.}\\ \noindent Assume
the following setup\,:\,\,\,$\nu_p=\frac{1}{p}$ \,\,with
probability $1$,\,where
$p\in\Delta=\left\{1,\frac{1}{2},\dots,\frac{1}{n},\dots\right\}$\,,
and so $\mathcal{P}_p(z)=z^{1/p}$\,. \par Clearly, the
corresponding semigroup $\mathcal{A}$ is commutative.\par
Furthermore, $\varphi(t)=\exp\{-t\}=\int\limits_0^\infty
e^{-tx}dA(x)$, where $A(x)$ is a cdf with a single unit-sized jump
at $x=1$. In this setup the strictly $\nu$-normal ch.f. is the
ch.f. of the normal (in the usual sense) r.v. with the zero mean.
\par

{\bf Example 2\,.}\;{\it The geometric summation scheme.}\\
\noindent
Suppose, $\nu_p$ is the r.v. having a geometric distribution
\[
\mathbf{P}\{\nu_p=k\}=p(1-p)^{k-1}\,,\,\,\,k=1,2,\dots\,\,\,,\,p\in(0,1)\,.
\]

Clearly, here $\mathbf{E}\left[\nu_p\right]=\frac{1}{p}$\,,\,\,and
$\mathcal{P}_p(z)=\frac{pz}{1-(1-p)z}$\,,\,\,$p\in(0,1)$\,. It is
quite straightforward to check that $\mathcal{A}$ is commutative.
\par Moreover, a direct calculation gives
$\varphi(t)=\frac{1}{1+t}=\int\limits_0^\infty e^{-tx}e^{-x}dx$,
i.e. $A(x)$ is the cdf of the exponential distribution. So that a
$\nu$-analogue of the strictly normal distribution is the Laplace
distribution with the ch.f.\, $f(t)=\frac{1}{1+at^2}$\,\,.

{\bf Example 3\,.}\;{\it Branching process scheme.}\\
\noindent
Let $\mathcal{P}(z)$ be some generating function, with
$\mathcal{P}^{\prime}(1)=\frac{1}{p_0}>1$ (so that the introduced
notation is $p_0=1/\mathcal{P}^{\prime}(1)$, with the condition
$p_0<1$).

Consider now a family given by
$\mathcal{P}^{0\,n}(z)=\mathcal{P}^{0(n-1)}\left(\mathcal{P}(z)\right)$\,,\,\,$n=1,2,\dots$\,.
Related to that is another family of the r.v.'s
$\nu_p$\,:\,\,$\mathcal{P}_p(z)=\mathcal{P}^{0\,n}(z)\,,\,\,p\in\left\{\frac{1}{p_0^n}\,,
\,\,n=1,2,\dots\right\}=:\Delta$\,.

Clearly, the semigroup $\mathcal{A}$ coincides with the family
$\{\mathcal{P}_p\,,\,\,p\in\Delta\}$\,.\,\,The ch.f. $\varphi(t)$
is a solution of the functional equation
$\varphi(t)=\mathcal{P}\left(\varphi(p_0t)\right)$\,\,.

It can be noted that the content of the paper by Mallows and Shepp
\cite{SheppMall1} is actually based on considering an example
identical to the {\it Example 3} above. Probably, neither the
authors of that work nor its reviewers were familiar with the
works by Klebanov and Rachev \cite{KlRa} and Bunge \cite{Bung},
which had dealt with exactly the same example a number of years
earlier.

Like mentioned in Introduction, in the present work we aim in
widening the collection of examples that involve random summation
with the commutative semigroup $\mathcal{A}$. For that purpose, we
address the description of pairs of certain commutative generating
functions $\mathcal{P}$ and $\mathcal{Q}$, i.e. the ones for which
the balance equality
$\mathcal{P}\circ\mathcal{Q}=\mathcal{Q}\circ\mathcal{P}$ holds,
-- but including only the case when {\it there exists no} such
function $\mathcal{H}$ such that $\mathcal{P}=\mathcal{H}^{0k}$
and $\mathcal{Q}=\mathcal{H}^{0m}$\,\,for some
$k,m\in\mathbb{N}$\,\,(which would be exactly the case of the {\it
Example 3}).

In a general setting, the problem of describing all such
commutative pairs of generating functions appears, unfortunately,
 far too involved to approach. However, certain special cases
are rather straightforward for consideration. In order to approach
the problem, we will use certain notions typical for the theory of
iterations of analytic functions, that we outline in the separate
section below.

\section{Theoretic justification via iterations of analytic functions}

Let $\mathcal{P}$ be a rational function with $(\deg)\geq2$.
Denote by $\mathcal{P}^{0n}$ its $n$th iteration. The functions
$\mathcal{P}$ and $\mathcal{Q}$ are called {\it conjugates}, if
there exists a linear-fractional function $R$, such that
$\mathcal{P}\circ R=R\circ\mathcal{Q}$\,.

A subset $E$ of the extended complex plane $\overline{\mathbb{C}}$
is called {\it completely invariant}, if its complete inverse image $\mathcal{P}^{-1}(E)$
coincides with $E$. The maximal finite completely invariant set $E(\mathcal{P})$ exists and
is called the {\it exceptional set} of the function
$\mathcal{P}$\,.\,It is always the case that $\mbox{card}\,
E(\mathcal{P})\leq2$\,. Moreover, if $\mbox{card}\,
E(\mathcal{P})=1$ then the function $\mathcal{P}$ is a conjugate
to a polynomial, while for $\mbox{card}\,E(\mathcal{P})=2$ the
function $\mathcal{P}$ is a conjugate to
$\mathcal{Q}(z)=z^n\,,\,\,n\in\mathbb{Z}\backslash\{0,1\}$\,.
Clearly, $E(\mathcal{Q})=\{0,\infty\}$.

If $\mathcal{P}$ is a rational function, then it is known (see e.g.
\cite{HeavyTailed}) that there is a finite number of open sets
$F_i$\,,\,\,\,$i=1,\dots,r$\,,\,\,which are {\it left invariant} by
the operator $\mathcal{P}$ and are such that (in the sequel, we will
refer to the two points below as {\it Conditions})
\begin{enumerate}
\item the union $\bigcup\limits_{i=1}^rF_i$ is {\it dense} on the
plane\,;
\item and $\mathcal{P}$ behaves {\it regularly} and in a unique
way on each of the sets $F_i$\,.
\end{enumerate}
The latter means that the termini of the sequences of iterations
generated by the points of $F_j$ are either precisely the same set,
which is then a finite cycle, or they are finite cycles of finite or
annular shaped sets that are lying concentrically. In the first case
the cycle is {\it attracting}, in the second one it is {\it
neutral}.

The sets $F_j$ are the {\it Fatou domains} of $\mathcal{P}$, and
their union is the {\it Fatou set}\, $F(\mathcal{P})$ of
$\mathcal{P}$.

The complement of $F(\mathcal{P})$ is the {\it Julia set}
\,$\mathcal{J}(\mathcal{P})$ of $\mathcal{P}$\,. Note that
$\mathcal{J}(\mathcal{P})$ is either a nowhere dense set (that is,
without interior points) and an uncountable set (of the same
cardinality as the real numbers), or
$\mathcal{J}(\mathcal{P})=\overline{\mathbb{C}}$\,. Like
$F(\mathcal{P})$, $\mathcal{J}(\mathcal{P})$ is left invariant by
$\mathcal{P}$, and on this set the iteration is {\it repelling},
meaning that
$\,\,|\,\mathcal{P}(z)-\mathcal{P}(w)\,|\,>\,|\,z-w\,|\,\,$ for
all elements $w$ in a neighborhood of $z$ (within
$\mathcal{J}(\mathcal{P})$). This means that $\mathcal{P}(z)$
behaves chaotically on the Julia set. Although there are points in
the Julia set whose sequence of iterations is finite, there is
only a countable number of such points (and they make up an
infinitely small part of the Julia set). The sequences generated
by points outside this set behave chaotically, a phenomenon called
{\it deterministic chaos}. Let $z_0$ be a repelling fixed point of
the function $\mathcal{P}$, and let
$\lambda=\mathcal{P}^{\prime}(z_0)$. Define
$\Lambda\,:\,\,z\rightarrow\lambda z$\,. Then there exists a
unique solution of the Poincar\'{e} equation
\[
F\circ\Lambda=\mathcal{P}\circ
F\,,\,\,\,\,F(0)=z_0\,,\,\,\,F^{\prime}(0)=1\,,
\]
that is meromorphic in $\overline{\mathbb{C}}$\,.

Now let
\[
\mathcal{I}(\mathcal{P})=F^{-1}(\mathcal{J}(\mathcal{P}))\,.
\]
If for two functions $\mathcal{P}$ and $\mathcal{Q}$ we have
$\mathcal{P}\circ\mathcal{Q}=\mathcal{Q}\circ\mathcal{P}$, then they
have the same function $F$.

\vspace*{0.2cm} There are the two following possibilities:
\vspace*{-0.2cm}
\begin{enumerate}
\item $\mathcal{I}(\mathcal{P})=\mathbb{C}$\,,\,\,in which case
$\mathcal{J}(\mathcal{P})=\overline{\mathbb{C}}$,.
\item $\mathcal{I}(\mathcal{P})$ is nowhere dense and consists of analytic cuvrves.
\end{enumerate}

Fatou \,\cite{Fatou}, and Julia \,\cite{Julia} investigated the
case. It turned out that is this case $\mathcal{P}$ and
$\mathcal{Q}$ can be reduced by a conjugancy either to the form
$\mathcal{P}(z)=z^m$ and $\mathcal{Q}(z)=z^n$ or to the form
$\mathcal{P}(z)=T_m(z)$ and $\mathcal{Q}(z)=T_n(z)$\,,\,where $T_k$
is the Chebyshev polynomial determined by the equation\,
$\cos(k\zeta)=T_k(\cos\zeta)$.

\section{Main results}

\subsection{A new example}

Let us return to the study of $\nu$-normal and $\nu$-stable random
variables. Recall that we deal with the family
$\{\nu_p\,,\,\,p\in\Delta\}$ taking its values in
$\mathbb{N}=\{1,2,\dots\}$\,. As before, we work with the
generating function\,,\,
$\mathcal{P}_p(z)=\mathbf{E}\left[\,z^{\nu_p}\,\right]$\,,\,of
$\nu_p$\,. The important result that we stressed says the a
strictly $\nu$-normal (resp. strictly $\nu$-stable) r.v. exist iff
the semigroup $\mathcal{A}$ generated by
$\{\mathcal{P}_p\,,\,\,p\in\Delta\}$\,\, is commutative. If
$\mathcal{P}_p\,,\,\,p\in\Delta$\,,\, is a rational function (with
$\deg\leq2$) satisfying {\it Condition 2}\, of the above section,
then either $\mathcal{P}_p(z)$ is reduced to a form
$\widetilde{\mathcal{P}}_p(z)=z^{1/p}$\,,\,\,$p\in\left\{\frac{1}{n}\,,\,n=1,2,\dots\right\}$,
and then we deal, in fact, with the classical (deterministic)
summation scheme, or $\mathcal{P}_p(z)$ is reduced to the form
$\mathcal{P}_p(z)=T_{1/\sqrt{p}}(z)$\,,\,\,$p\in\left\{\frac{1}{n^2}\,,\,n=1,2,\dots\right\}$\,.
Clearly, the polynomial $T_m(z)$ is not a generating function
itself, however a function to which it is a conjugate,
specifically the function
\begin{equation}\label{Cheb}
\mathcal{P}_p(z)=\frac{1}{T_{1/\sqrt{p}}(z)}\,\,,\,\,\,p\in\left\{\frac{1}{n^2}\,,\,
n=1,2,\dots\right\}\,\,,
\end{equation}
is indeed a generating function, -- the fact that we prove below.
Moreover, below we consider in some details a family of r.v.'s
$\left\{\nu_p\,,\,p\in\left\{\frac{1}{n}\,,\,n=1,2,\dots\right\}\,\right\}$
that have generating functions of the form (\ref{Cheb}), and
investigate the corresponding strictly $\nu$-normal and strictly
$\nu$-stable distributions.

\begin{lm}\label{Lem1}
Let $P_n(x)$ be a polynomial with \,$\deg P_n=k$\,\, by the even
powers of $x$, and whose zeros are all within the interval
$(-1,1)$\,.\,\,Let \,$P_n(1)=1$ and polynomial's coefficient with
with $x^n$ be positive. Then for any natural number $k$, the
function
\[
\mathcal{P}(x)=\frac{x^k}{P_n(\frac{1}{x})}
\]
is a generating function.
\end{lm}

\begin{proof}
Represent $P_n(x)$ as
\[ P_n(x)=b_0+b_1x+\dots+b_nx^n=b_n\prod\limits_{j=1}^n\left(x-a_j\right)\,, \]
where $a_j$ \,($j=1,\dots,n$)\,\,are the zeros of the polynomial
$P_n$\,\,\,sorted in the order of ascendance. As $P_n$ is a
polynomial by the even powers of $x$, then, if $a_j$ is a zero of
$P_n$, then $-a_j$ is also a zero of $P_n$. Therefore,
\begin{eqnarray}\label{Series}
\nonumber \frac{1}{P_n(\frac{1}{x})}&=&\frac{1}{b_n\prod\limits_{j=1}^n\left(\frac{1}{x}-a_j\right)}\\
\nonumber
&=&\frac{1}{b_n\prod\limits_{j=1}^{n/2}\left(\frac{1}{x}-a_j\right)\left(\frac{1}{x}+a_j\right)}\\
\nonumber &=&\frac{1}{b_n}\prod\limits_{j=1}^{n/2}\frac{1}{\left(\frac{1}{x}-a_j\right)
\left(\frac{1}{x}+a_j\right)}\\
&=&\frac{1}{b_n}\prod\limits_{j=1}^{n/2}\frac{x^2}{1-a_j^2x^2}
\end{eqnarray}
Obviously,
\[
\frac{x^2}{1-a_j^2x^2}=\sum\limits_{k=0}^{\infty}a^{2k}_jx^{2k+2}
\]
is a series with positive (non-negative) coefficients, converging
when $|x|\leq1$. From (\ref{Series}), it now follows that\,
$\mathcal{P}(x)=\frac{x^k}{P_n(\frac{1}{x})}$\, is a series also
convergent when $|x|\leq1$, having non-negative coefficients, and
$\mathcal{P}(1)=1$\,. Hence, $\mathcal{P}(x)$ is a generating
function of some random variable.
\end{proof}

\begin{co}
Let $T_n(x)$ be a Chebyshev polynomial of degree $n$. Then
\[ \mathcal{P}(x)=\frac{1}{T_n(\frac{1}{x})} \]
is a generating function of some r.v. taking values in
$\mathbb{N}$.
\end{co}

\begin{proof}
When $n$ is an even number, the result follows directly from {\it
Lemma \ref{Lem1}} and from the properties of Chebyshev
polynomials. For odd $n$, consider the representation \\
$T_n(x)=xP_{n-1}(x)$, where $P_{n-1}(x)$ is a polynomial by the
even degrees of $x$, satisfying the conditions of {\it Lemma
\ref{Lem1}}.
\end{proof}

Let us now set
$\Delta:=\left\{\frac{1}{n^2}\,,\,n=1,2,\dots\right\}$\,.\,\,Consider
the family of generating functions
\[
\mathcal{P}_p(z)=\frac{1}{T_{1/\sqrt{p}}(1/z)}\,,\,\,\,\,p\in\Delta\,.
\]
Clearly, \,\,
$\mathcal{P}_{p_1}\circ\mathcal{P}_{p_2}=\mathcal{P}_{p_2}\circ\mathcal{P}_{p_1}$\,\,
for all $p_1,p_2\in\Delta\,$\,, due to the well known property of
Chebyshev polynomials stating that\,\,$T_n(T_m(x))=T_{n\cdot
m}(x)$.\, In other words, semigroup generated by the family
$\{\,\mathcal{P}_p\,,\,\,p\in\Delta\}$ is commutative. It follows
(see e.g. \cite{HeavyTailed}) that there exists a solution to the
system of equations
\begin{equation}\label{Syst}
\varphi(t)=\mathcal{P}_p(\varphi(pt))\,\,,\,\,\,p\in\Delta\,,
\end{equation}
satisfying initial conditions
\begin{equation}\label{Cond}
\varphi(0)=1\,\,,\,\,\,\varphi^{\prime}(0)=-1\,,
\end{equation}
and the solution is unique.

Since $T_n(x)=\cos(n\cdot\arccos
x)=\cosh(n\cdot\mbox{arccosh}\,x)$\,, the direct plugging gives that
the function
\begin{equation}\label{PhiCosh}
\varphi(t)=1/\cosh\left(\sqrt{2t}\right)
\end{equation}
satisfies the system (\ref{Syst}), as well as the conditions
(\ref{Cond}). Hence, the function
\begin{equation}\label{Norm}
f(t)=\frac{1}{\cosh(at)}\,,\,\,\,a>0
\end{equation}
is actually a ch.f. of a strictly $\nu$-normal r.v.. The ch.f.
(\ref{Norm}) is, in fact, well known -- it is the ch.f. of the
{\it hyperbolic secant distribution}. Clearly,  here \,$a$\, is
the scale parameter. When $a=1$, it is the case of the {\it
standard hyperbolic secant distribution}, whose pdf has the form
\[
p(x)=\frac{1}{2}\,\mbox{sech}\,\left(\frac{\pi x}{2}\right)\,,
\]
while the cdf is
\[
F(x)=\frac{2}{\pi}\,\mbox{arctan}\,\left[\exp\left(\frac{\pi
x}{2}\right)\right]\,.
\]
Furthermore, in order to obtain the expression for the ch.f. of
strictly $\nu$-stable distributions, one just needs to apply the
relation (\ref{InfDiv}) to the strictly stable (in the usual
sense) ch.f. $h$.

\subsection{An interesting property}

Note that the function $\varphi$, as represented by
(\ref{PhiCosh}), can be viewed somewhat interesting on its own,
and so we shall address its properties and consider its cdf
$A(x)$\,(which corresponds to $\varphi(t)$ via (\ref{O5})\,\,)\,.

Let $W_1(t)$ and $W_2(t)$\,,\,\,$t\geq0$\,,\,\,be two independent
Wiener processes. Consider a r.v.
\begin{equation}\label{Xxi}
\xi = \int\limits_0^1W_1^2(t)dt+\int\limits_0^1W_2^2(t)dt\,.
\end{equation}
This r.v. is well studied, and it is known (see e.g. \cite{Talacko})
that its Laplace transform equals to
\[
\mathbf{E}\,\left[\,e^{-t\xi}\,\right]=\frac{1}{\cosh\left(\sqrt{2t}\right)}\,,
\]
which coincides with $\varphi(t)$ as given by (\ref{PhiCosh}).

Hence $A(x)$ is the cdf of the r.v. $\xi$. On the other hand, as
follows from Gnedenko's Transfer Theorem,
\[
A(x)=\underset{p\rightarrow0}
{\lim}\,\,\mathbf{P}\,\{\,p\,\nu_p<x\,\}\,\,.
\]

Consequently, the following theorem is valid.

\begin{thm}\label{th41}
Let $\left\{\,\nu_p<x\,,\,\,p\in\Delta\,\right\}$ be a family of
r.v.'s having generating functions
\[
\mathcal{P}_p(z)=\frac{1}{T_{1/\sqrt{p}}(\frac{1}{z})}\,,\,\,\,\,p\in\Delta=
\left\{\frac{1}{n^2}\,,\,n=1,2,\dots\right\}\,.
\]
Then
\[ \underset{p\rightarrow0}
{\lim}\,\,\mathbf{P}\,\{\,p\,\nu_p<x\,\}=\mathbf{P}\{\xi<x\}\,, \]
where the r.v. $\xi$ is the one defined via (\ref{Xxi})\,.
\end{thm}

Theorem \ref{th41} may be reformulated in the following way.
\par
{\it Let
\[ \frac{1}{T_{n}(\frac{1}{z})}= \sum_{k=0}^{\infty}p_k(n)z^k.\]
Then
\[ \lim_{n \to \infty} \sum_{k=0}^{[n^2x]} p_k(n) = \mathbf{P}\{\xi<x\}.  \]
}

\begin{figure}
\begin{center}
\includegraphics[totalheight=6cm]{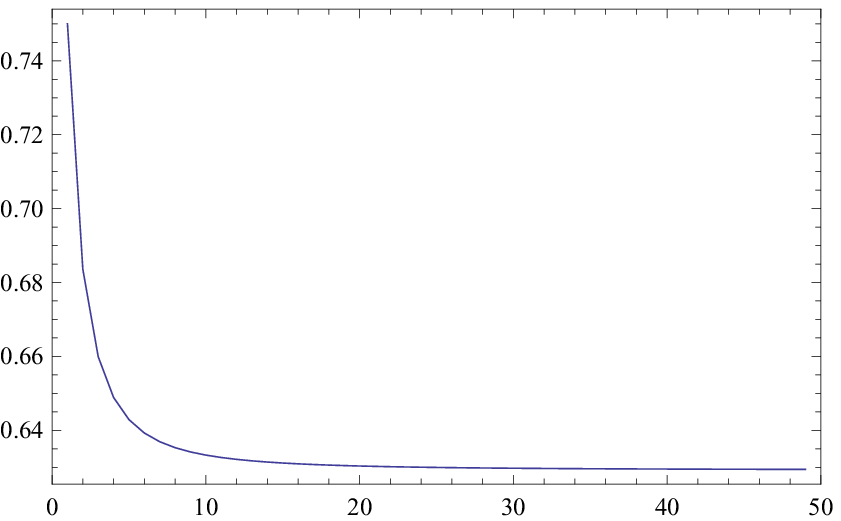}
\end{center}
\label{fig1}
\caption{Plot of the $\sum_{k=0}^{[n^2x]} p_k(n) $ as the function of $n=2, \ldots , 50$}
\end{figure}

On Figure 1, the plot of the $\sum_{k=0}^{[n^2x]} p_k(n)$ is given
as a function of $n$ starting with $n=2$ until $n=50$. We see that
the functions attains the constant level rather quickly, and
therefore it is possible to use the asymptotic result for $n>25$.

\begin{co}
Let $X$ be a r.v. having the standard hyperbolic secant
distribution. Then its distribution can be represented in the form
of a scale mixture of a normal distribution with zero mean and
standard deviation $\sqrt{\xi}$, where $\xi$ is defined via
(\ref{Xxi})\,.
\end{co}
To prove the above, one just needs to write the ch.f. of $X$ in the
form $\int\limits_0^\infty e^{-t^2x}dA(x)$\,,\,\,and note that
$e^{-t^2x}$ is actually the ch.f. of the standard Normal r.v.
$N(0,\sigma^2)$\,\,($\sigma^2=x$), while $A(x)$ is the cdf of
$\xi$\,.

Note that there is a certain analogy between the representation
$A(x)$ as the cdf of the r.v. $\xi$ from (\ref{Xxi}) and the
corresponding result in the scheme of the random summation with
geometric distribution (see e.g. \cite{HeavyTailed}). Specifically,
considering the family $\left\{\nu_p\,,\,\,p\in(0,1)\right\}$ having
the geometric distribution
$\mathbf{P}\,\{\,\nu_p=k\,\}=p(1-p)^{k-1}$\,,\,\,$k=1,2,\dots$\,,
 the function\,\,\,$\varphi$\,\,\,turns into
\[
\varphi(t)=\frac{1}{1+t}=\int\limits_0^\infty e^{-tx}dA_1(x)\,,
\]
where $A_1(x)$ is the cdf of the exponential
distribution,\,\,i.e.\,\,
$A_1(x)=1-e^{-x}$\,\,for\,\,$x>0$\,\,\,and
$A_1=0$\,\,for\,\,$x\leq0$\,. It can be checked that if $\eta_1$
and $\eta_2$ are two independent standard Normal r.v.'s, then
$A_1$ is a cdf of the r.v.
\,\,$\xi_1=\eta^2_1+\eta^2_2$\,\,,\,\,which is, in a way, related
to (\ref{Xxi}).

\subsection{Characterizations}

Let us now turn to the characterizations of the distribution of the
r.v. (\ref{Xxi}) and of the hyperbolic secant distribution.

\begin{thm}
Let $X_1,\dots,X_n,\dots$ be a sequence of non-negative iid random
variables, and
$\nu_p\,,\,\,p\in\left\{\frac{1}{n^2}\,,\,n=2,\dots\right\}$\,, is
a family of the r.v.'s having the generating function\,\,
$\mathcal{P}_p(z)=\frac{1}{T_{1/\sqrt{p}}(\frac{1}{z})}$\,,\,\,independent
of the sequence\,\, $\{X_j\,,\,\,j\geq1\}$\,.

If, for some fixed $p\in\Delta$\,,
\begin{equation}\label{X1d}
X_1\overset{d}{=}\,p\sum\limits_{j=1}^{\nu_p}X_j\,\,,
\end{equation}
(where "$\overset{d}{=}$" is the equality in distribution), then
$X_1$ has the distribution whose Laplace transform is
\begin{equation}\label{XLapl}
\mathbf{E}\left[\,e^{-tX}\,\right]=\frac{1}{\cosh\left(\sqrt{at}\right)}\,,\,\,a>0\,.
\end{equation}
\end{thm}

\begin{proof}

The equality (\ref{X1d}), in terms of the Laplace transform
$\Psi(t)=\mathbf{E}\,e^{-tX}$\,,\,\,can be represented as
\begin{equation}\label{PhiLap}
\Psi(t)=\mathcal{P}_p\left(\Psi(pt)\right)\,.
\end{equation}
Clearly, the function
\[
\Psi_a(t)=\frac{1}{\cosh\left(\sqrt{at}\right)}
\]
satisfies (\ref{PhiLap}) for any $a>0$ and, moreover, is analytic
in the strip $\,|\,t\,|\,<r$\,\,(\,$r>0$\,)\,.

In the following, we use the results of the book by Kakosyan,
Klebanov and Melamed \cite{KKM}.
Example 1.3.2 of this book shows that $\{\,\Psi_a\,,\,\,a>0\,\}$
forms a strongly $\varepsilon$-positive family, where
$\varepsilon$ is a set of restrictions of Laplace transforms of
probability distributions given in $R_+$ on an interval \,$[0,T]$
\,$(0<T<r)$\,.

Clearly, the operator
\,\,$A\,:\,f\,\rightarrow\,\mathcal{P}_p(f(pt))$ on $\varepsilon$ is
intensively monotone.

The result follows from Theorem 1.1.1 of the above mentioned book
(page 2).
\end{proof}

\begin{thm}
Let $X_1,\dots,X_n,\dots$ be a sequence of non-negative iid random
variables,  having a symmetric distribution, while
$\left\{\nu_p\,,\,\,p\in\Delta\right\}$ is the same family as in the
previous Theorem.

If, for some fixed $p\in\Delta$\,,
\begin{equation}\label{X1d2}
X_1\overset{d}{=}\,p^{1/2}\sum\limits_{j=1}^{\nu_p}X_j\,\,,
\end{equation}
then $X_1$ has the hyperbolic secant distribution whose ch.f. is
\begin{equation}\label{XLap2}
f(t)=\frac{1}{\cosh\left(at\right)}\,,\,\,a>0\,.
\end{equation}
\end{thm}

\begin{proof}

Quite analogous to the proof of the previous Theorem, with the
difference that instead of Example 1.3.2, the use of the Example
1.3.1 from \cite{KKM} is sufficient.
\end{proof}

\section{Other examples}

There exist examples of the pairs of commutative functions, which
are {\it not rational}. Here we refer to the two classes of such
functions, the first of which was investigated by Melamed \cite{MLN} and the second appears at first in the present work.

{\bf Example I\,.}\,\,\, (See Melamed \cite{MLN}  for detailed study) \par

Consider the family of generating functions
\begin{equation}\label{Ex1}
\mathcal{P}_p(z)=\frac{p^{1/m}\,z}{(1-(1-p)\,z^m)^{1/m}}\,\,,
\end{equation}
where \,\,$p\in(0,1)$\,,\,\,and $m$ is a fixed positive integer.
Obviously, in the case $m=1$, $\mathcal{P}_p(z)$ reduces to the
generating function of the geometric distribution, and has already
been mentioned this case above. Hence, assume that $m\geq2$. In
that case, it is easy to check that
\begin{equation}\label{Ex1-2}
\varphi(t)=\frac{1}{(1+mt)^{1/m}}\,\,,
\end{equation}
and therefore the ch.f. of the strictly $\nu$-normal distribution
(for the family $\left\{\nu_p\,,\,\,p\in\Delta\right\}$ \,having
the generating function (\ref{Ex1})\,) has the form
\[
\varphi(t)=\frac{1}{(1+mat^2)^{1/m}}\,\,,
\]
with a parameter $a>0$\,.

\vspace*{0.4cm}

{\bf Example II\,.}\,\,\par

Consider the family of functions
\begin{equation}\label{Ex2}
\mathcal{P}_p(z)=\frac{1}{\left(T_{1/\sqrt{p}}\left(\frac{1}{z^m}\right)\right)^{1/m}}\,\,,
\end{equation}
where
\,\,$p\in\left\{\frac{1}{n^2}\,,\,n=2,\dots\right\}$\,,\,\,and
$m\geq1$\,(an integer)\,.

Using a slightly modified version of the proof of {\it Lemma 1},
it is easy to check that $\mathcal{P}_p(z)$ is a generating
function of some r.v.
$\nu_p\,,\,\,p\in\left\{\frac{1}{n^2}\,,\,n=2,\dots\right\}$ for
any fixed whole number $m\geq1$ (surely, both
$\mathcal{P}_p$\,\,\,and\,\,$\nu_p$ both depend on $m$, but we
omit this dependence in the notation).

The case $m=1$ has already been considered above. For $m\geq2$
analogous methods are applicable, and so will refer to the results
only. Specifically,
\begin{equation}\label{Ex2-2}
\varphi(t)=\frac{1}{\left(\cosh\sqrt{2mt}\right)^{1/m}}\,\,,
\end{equation}
while the ch.f. of the corresponding strictly $\nu$-normal
distribution has the form
\begin{equation}\label{Ex2-3}
f(t)=\frac{1}{\left(\cosh at\right)^{1/m}}\,\,,
\end{equation}
where $a>0$\,.

Note that in the case $m=2$, we have the following expressions for
the distributions whose Laplace transforms are (\ref{Ex1-2}) and
(\ref{Ex2-2}).

For $m=2$, the formula (\ref{Ex1-2}) gives
\[
\varphi(t)=\frac{1}{\sqrt{1+2t}}\,\,.
\]
This function is the Laplace transform of the distribution of the
r.v. $X^2$\,,\,with $X$ being the standard Normal r.v.

In a similar way, (\ref{Ex2-2}) gives for $m=2$
\[
\varphi(t)=\frac{1}{\sqrt{\cosh\sqrt{4t}}}\,\,.
\]

This function is the Laplace transform of the distribution of the
r.v. $I=\int\limits_0^1X^2(t)dt$\,,\,\,where $X(t)$ is the standard
Wiener process.

\end{document}